\documentclass[12pt]{article}

\usepackage{amsmath, amssymb, amsfonts, amsbsy, amscd, latexsym, amsthm}
\usepackage{enumitem}
\date{}

\newtheorem{Theorem}{Theorem}[section]
\newtheorem{Proposition}[Theorem]{Proposition}
\newtheorem{Corollary}[Theorem]{Corollary}

\theoremstyle{definition}

\theoremstyle{remark}
\newtheorem{Remark}[Theorem]{Remark}

\numberwithin{equation}{section}

\title{Extensions of classical hypergeometric identities of Bailey and Whipple}

\author{Ilia D. Mishev
\footnote{Department of Mathematics, University of Colorado Boulder,
Campus Box 395, Boulder, CO 80309-0395, U.S.A.
E-mail address: ilia.mishev@colorado.edu}}

\begin{document}

\maketitle

\begin{center}
Dedicated to the memory of Professor Plamen Djakov.
\end{center}

\begin{abstract}
We obtain extensions of classical hypergeometric identities of Bailey
and Whipple that transform nearly-poised and very-well-poised series to
Saalsch\"utzian series, Saalsch\"utzian series to Saalsch\"utzian series, and
very-well-poised and nearly-poised series to very-well-poised series.
We employ a method in which summations and transformations of lower-order
series are used to obtain transformations of higher-order series.
By taking limits, we also obtain extensions of two classical
quadratic transformations of Whipple and Bailey.
Furthermore, we show how a number of other well-known results regarding
hypergeometric series follow as special
cases of our results.
\end{abstract}

\section{Introduction}

Identities among hypergeometric series, both terminating and nonterminating, 
have been the subject of
extensive research. In the 1920s, 
Bailey and Whipple 
(see \cite{Bailey3, Whipple1,Whipple2, Whipple3, Whipple4, Whipple5})
found a 
number of identities relating various terminating hypergeometric series
that are listed in Chapters 4 and 7 of Bailey's tract \cite{Ba}. In this paper,
we generalize these terminating hypergeometric identities found by 
Bailey and Whipple and extend them to higher-order hypergeometric series.

We consider classical identities involving terminating very-well-poised, nearly-poised 
and Saalsch\"utzian series (see Section 2 for the relevant definitions).
The highest order classical transformation formula between two terminating very-well-poised
series was found by Bailey in \cite{Bailey3} and involves two terminating very-well-poised
${}_9F_8(1)$ series (see \cite[Eq.\ 4.3.7]{Ba}). 
This formula was recently extended to the ${}_{11}F_{10}$ level
by Srivastava, Vyas and Fatawat \cite[Theorem 3.4]{SVF} by adding two new numerator and
denominator pairs of parameters that differ by one each. In this paper, we further extend
Srivastava, Vyas and Fatawat's result to a relation involving two terminating
very-well-poised ${}_{13}F_{12}(1)$ series 
(see Proposition \ref{13P3} in Section 5) by adding two more
pairs of numerator and denominator parameters with unit difference. 

In Section 4.5 of Bailey's tract \cite{Ba}, one can find a number of transformation formulas
involving terminating nearly-poised series discovered by Whipple and Bailey. 
The identities \cite[Eq.\ 4.5.1]{Ba} (originally found by Whipple in \cite{Whipple5}) and
\cite[Eq.\ 4.5.2]{Ba} (originally found by Bailey in \cite{Bailey3}) relate nearly-poised
${}_4F_3(1)$ and ${}_5F_4(1)$  series, respectively, to Saalsch\"utzian ${}_5F_4(1)$ series.
In this paper (see Corollary \ref{2C3P16} in Section 3),
we extend these two identities to a single transformation formula between a ${}_5F_4(1)$ series,
which is {\it not} nearly-poised but in which {\it two} pairs of numerator and denominator parameters 
deviate from ``well-poisedness", and a Saalsch\"utzian ${}_6F_5(1)$ series. 
In addition, this extension is further generalized by Proposition \ref{3P16} in Section 3 to a relation between
two terminating ${}_7F_6(1)$ series, one of which is Saalsch\"utzian. A special case of Proposition \ref{3P16}
given  in Corollary \ref{1C3P16}
coincides with a recent special case of a result of Maier \cite{Maier}.

The classical Whipple transformation between a very-well-poised ${}_7F_6(1)$ and a 
Saalsch\"utzian ${}_4F_3(1)$ series 
(see \cite{Whipple2}, \cite{Whipple4} and \cite[Eq.\ 4.3.4]{Ba})
was recently generalized by Srivastava, Vyas and Fatawat \cite[Theorem 3.2]{SVF}.
In this paper, we provide a very general result in Proposition \ref{11P2} that extends both 
 Srivastava, Vyas and Fatawat's result \cite[Theorem 3.2]{SVF} and the results described in the
paragraph above.

In Section 4 of this paper, we study transformations between two terminating Saalsch\"utzian series.
The classical result in this area is the Whipple transform (see \cite{Whipple2}, \cite{Whipple3} and \cite[Eq.\ 7.2.1]{Ba}) 
involving two terminating Saalsch\"utzian
${}_4F_3(1)$ series. In Proposition \ref{12P2} in Section 4, we extend the Whipple transform to a transformation
that involves  two terminating Saalsch\"utzian
${}_6F_5(1)$ series in each of which series 
two numerator parameters exceed two 
denominator parameters by one.

In Section 5, in addition to obtaining the above-mentioned relation involving two terminating
very-well-poised ${}_{13}F_{12}(1)$ series, we also obtain extensions of
classical results found by Bailey in 
\cite[Eqs.\ 8.1, 8.2 and 8.3]{Bailey3}
and reproduced in
\cite[Eqs.\ 4.5.3, 4.5.4 and 4.5.5]{Ba}
that transform terminating nearly-poised ${}_5F_4(1)$ series
with parametric excesses $\omega=1$ and $\omega=2$ and
a terminating nearly-poised ${}_6F_5(1)$ series 
with parametric excess $\omega=1$ to
terminating very-well-poised ${}_9F_8(1)$ series. 
The immediate extension of \cite[Eqs.\ 4.5.3, 4.5.4 and 4.5.5]{Ba}
is given in Corollary \ref{1C13P4}, and a further extension is provided in 
Proposition \ref{13P4}.

To obtain our results, we use a method employed by Bailey in \cite{Bailey3} and
\cite[Chapter 4]{Ba} that utilizes sums of series of lower order to obtain transformations
of series of higher order. 
The extensions of Bailey's general formulas \cite[Eqs.\ 4.3.1 and 4.3.6]{Ba} are provided in
Propositions \ref{11P1} and \ref{13P1}, respectively. 
We should point out that Bailey uses the Pfaff--Saalsch\"utz formula 
(see \cite[Eq.\ 2.2.1]{Ba}) to obtain \cite[Eq.\ 4.3.1]{Ba} while we use the
extension of the Pfaff--Saalsch\"utz formula to a ${}_4F_3(1)$ series
given by Rakha and Rathie \cite{RR} to obtain the more general
Proposition \ref{11P1}. Moreover, Bailey uses Dougall's theorem 
(see \cite[Eq.\ 6]{Dougall}, \cite{Hardy2} and \cite[Eq.\ 4.3.5]{Ba}) to obtain
\cite[Eq.\ 4.3.6]{Ba} and we use the extension of Dougall's theorem to
a ${}_9F_8(1)$ series summation provided by Srivastava, Vyas and Fatawat 
\cite[Theorem 3.3]{SVF} to prove the more general Proposition \ref{13P1}.
Finally, in this paper, 
not only do we use 
the method just described with
known summations of series, 
but we also use it with a known {\it transformation} of series and then reverse the
order of summation to obtain a new transformation (see Proposition \ref{11P2}).

The method described above that we use in this paper is parallel to the Bailey's transform
(see \cite{BaileyRR1}, \cite{BaileyRR2} and \cite[pp.\ 58--74]{Slater}), 
which is employed by Srivastava, Vyas and Fatawat in \cite{SVF}. Both methods
can be utilized to obtain many of the known transformations of hypergeometric series. 
We should also note that, according to the Karlsson--Minton summation formula
(see \cite{Karlsson}), any hypergeometric series in which numerator and denominator
parameters differ by positive integers can be written as a finite sum of hypergeometric
series of lower order, but we have not used this approach or formula in our paper.

Finally, in Section 6, by taking certain limits of relations in Section 3, we obtain
extensions of some classical quadratic transformations of hypergeometric series.
In particular, we obtain extensions in terms of single transformations of both
the quadratic transformation found by Whipple in \cite[Eq.\ $(7.1)$]{Whipple5}
and its companion transformation found by Bailey in \cite[Eq.\ $(9.1)$]{Bailey3}.
Also,
a special case of our most general extension there (see (\ref{1e6P1}) in Section 6
below) coincides with a special case of a transformation of 
Maier \cite[Theorem $3.4$]{Maier}.

\section{Preliminaries}

The hypergeometric series of type ${}_{r}F_s$ is defined by
\begin{equation}
\label{210}
{}_{r}F_s \left[ {\displaystyle a_1,a_2,\ldots,a_{r};
\atop \displaystyle b_1,b_2,\ldots,b_s;} z\right] =
\sum_{n=0}^{\infty} \frac{(a_1)_n(a_2)_n \cdots
(a_{r})_n}{n!(b_1)_n(b_2)_n \cdots (b_s)_n}z^n,
\end{equation}
where $r$ and $s$ are nonnegative integers, $a_1,a_2,\ldots,a_{r},
b_1,b_2,\ldots,b_s, z\in \mathbb{C}$, and the rising factorial
$(a)_n$ is given by
\begin{equation*}
(a)_n=\left\{ \begin{array}{ll}
a(a+1)\cdots(a+n-1), & n>0,\\
1, & n=0.
\end{array} \right.
\end{equation*}

In this paper we will be mostly interested in the case where $r=s+1$.
The series of type ${}_{s+1}F_s$ converges absolutely if $|z|<1$ or if
$|z|=1$ and $\textrm{Re}(\sum_{i=1}^sb_i-\sum_{i=1}^{s+1}a_i)>0$ (see
\cite[p.\ 8]{Ba}). We assume that no denominator parameter
$b_1,b_2,\ldots,b_s$ is a negative integer or zero. If a numerator
parameter $a_1,a_2,\ldots,a_{s+1}$ is a negative integer or zero, the
series has only finitely many nonzero terms and is said to 
{\it terminate}.

When $z=1$, we say that the series is of {\it unit argument} and of type
${}_{s+1}F_s(1)$. If $\sum_{i=1}^sb_i-\sum_{i=1}^{s+1}a_i=1$, the
series is called {\it Saalsch\"utzian}. If
$1+a_1=b_1+a_2=\cdots=b_s+a_{s+1}$, the series is called 
{\it well-poised}.
A well-poised series that satisfies $a_2=1+\frac{1}{2}a_1$ is called
{\it very-well-poised}.
The {\it parametric excess} $\omega$ is given by
$\omega=\sum_{i=1}^sb_i-\sum_{i=1}^{s+1}a_i$.
Note that $\omega=1$ for a Saalsch\"utzian series.

We will use the following extension of the classical Chu--Vandermonde formula
(see \cite[Section 1.3]{Ba}), which extension sums a special terminating
${}_3F_2(1)$ series where a numerator parameter exceeds a denominator
parameter by one:
\begin{equation}
\label{1e3f2}
{}_3F_2 \left( \left. 
{\displaystyle a,p+1,-n
\atop \displaystyle b,p}\right| 
1\right)
=\frac{(b-a-1)_n(q+1)_n}{(b)_n(q)_n},
\end{equation}
where
\begin{equation}
\label{2e3f2}
q=\frac{p(b-a-1)}{p-a}.
\end{equation}

The above formula (\ref{1e3f2}) appears in \cite{Miller}. A nonterminating version
of the same formula can be found in \cite[p.\ 534, Eq.\ (10)]{PBM}.
Letting $p \to \infty$ in (\ref{1e3f2}) yields the Chu--Vandermonde formula.

We will also use the following extension of the Pfaff--Saalsch\"utz formula 
(see \cite[Eq.\ 2.2.1]{Ba}) given by Rakha and Rathie \cite{RR} which
finds the sum of a special terminating Saalsch\"utzian ${}_4F_3(1)$ series
where a numerator parameter exceeds a denominator parameter by one:
\begin{eqnarray}
\label{e1R1C1P18}
&&{}_4F_3 \left( \left. 
{\displaystyle a,b,p+1,-n
\atop \displaystyle c,p,2+a+b-c-n}\right| 
1\right)\nonumber\\
&&=\frac{(c-a-1)_n(c-b-1)_n(q+1)_n}
{(c)_n(c-a-b-1)_n(q)_n},
\end{eqnarray}  
where
\begin{equation}
\label{e2R1C1P18}
q=\frac{p(c-a-1)(c-b-1)}
{ab+p(c-a-b-1)}.
\end{equation} 

Letting $p=b$ in (\ref{e1R1C1P18}) yields the Pfaff--Saalsch\"utz formula,
while letting $b \to \infty$ in (\ref{e1R1C1P18}) gives (\ref{1e3f2}).

We note that (\ref{e1R1C1P18}) can also be written as
\begin{eqnarray}
\label{e3R1C1P18}
&&{}_4F_3 \left( \left. 
{\displaystyle a-b-c,\gamma_1+1,a+n,-n
\atop \displaystyle 1+a-b,1+a-c,\gamma_1}\right| 
1\right)\nonumber\\
&&=\frac{(b)_n(c)_n(a-p+1)_n(p+1)_n}
{(1+a-b)_n(1+a-c)_n(p)_n(a-p)_n},
\end{eqnarray}  
where
\begin{equation}
\label{e4R1C1P18}
\gamma_1=\frac{p(a-p)(b+c-a)}
{bc-p(a-p)},
\end{equation} 
and as
\begin{eqnarray}
\label{e5R1C1P18}
&&{}_4F_3 \left( \left. 
{\displaystyle c-a-1,c-b-1,\gamma_2+1,-n
\atop \displaystyle c,\gamma_2,c-a-b-n}\right| 
1\right)\nonumber\\
&&=\frac{(a)_n(b)_n(p+1)_n}
{(c)_n(1+a+b-c)_n(p)_n},
\end{eqnarray}  
where
\begin{equation}
\label{e6R1C1P18}
\gamma_2=\frac{p(c-a-1)(c-b-1)}
{ab+p(c-a-b-1)}.
\end{equation} 

We will use (\ref{e3R1C1P18}) and (\ref{e5R1C1P18})
in Sections 3 and 4, respectively, where we study extensions of 
transformations of nearly-poised and very-well-poised series to
Saalsch\"utzian series and extensions of transformations of
Saalsch\"utzian series to Saalsch\"utzian series.

In \cite[Theorem 3.3]{SVF}, 
Srivastava, Vyas and Fatawat find a generalization of the classical Dougall's theorem
for the sum of a terminating very-well-poised ${}_7F_6(1)$ series with
parametric excess $\omega=2$ 
(see \cite[Eq.\ 6]{Dougall}, \cite{Hardy2} and \cite[Eq.\ 4.3.5]{Ba}). 
The generalization found by Srivastava, Vyas and Fatawat can be written as
\begin{eqnarray}
\label{1e1R3C11P2}
&&{}_9F_8 \left(  
{a,1+\frac{a}{2},b,c,d,
\atop 
\frac{a}{2},1+a-b,1+a-c,1+a-d,}
\right.\\
&&\left. \left. 
{2a-b-c-d+n,a-p+1,p+1,-n;
\atop 
1+b+c+d-a-n,p,a-p,1+a+n;}
\right| 1\right)
\nonumber\\
&&=
\frac{(1+a)_n(a-b-c)_n(a-b-d)_n(a-c-d)_n(\alpha+1)_n}
{(1+a-b)_n(1+a-c)_n(1+a-d)_n(a-b-c-d)_n(\alpha)_n},\nonumber
\end{eqnarray}
where
\begin{equation}
\label{2e1R3C11P2}
\alpha
=\frac{p(a-p)(a-b-c)(a-b-d)(a-c-d)}
{(2a-b-c-d+n)(bcd+p(a-p)(a-b-c-d))}.
\end{equation}

Dougall's theorem follows from (\ref{1e1R3C11P2})
by letting $p=b$.
We note that (\ref{1e1R3C11P2}) can also be written as
\begin{eqnarray}
\label{1e2R3C11P2}
&&{}_9F_8 \left(  
{\lambda,1+\frac{\lambda}{2},\lambda+b-a,\lambda+c-a,\lambda+d-a,
\atop 
\frac{\lambda}{2},1+a-b,1+a-c,1+a-d,}
\right.\\
&&\left. \left. 
{a+n,\frac{\lambda}{2}-\gamma+1,\frac{\lambda}{2}+\gamma+1,-n;
\atop 
1+\lambda-a-n,\frac{\lambda}{2}+\gamma,\frac{\lambda}{2}-\gamma,1+\lambda+n;}
\right| 1\right)
\nonumber\\
&&=
\frac{(1+\lambda)_n(b)_n(c)_n(d)_n(a-p+1)_n(p+1)_n}
{(a-\lambda)_n(1+a-b)_n(1+a-c)_n(1+a-d)_n(p)_n(a-p)_n},\nonumber
\end{eqnarray}
where
\begin{equation}
\label{2e2R3C11P2}
\lambda=2a-b-c-d
\end{equation}
and
\begin{equation}
\label{3e2R3C11P2}
\gamma^2
=\frac{\lambda^2}{4}
-\frac{p(a-p)(a-b-c)(a-b-d)(a-c-d)}
{bcd+p(a-p)(a-b-c-d)}.
\end{equation}

We will use (\ref{1e2R3C11P2}) in Section 5
where we study extensions of transformations of
very-well-poised and nearly-poised series to very-well-poised series.

\section{Extensions of hypergeometric transformations of nearly-poised and very-well-poised 
series to
Saalsch\"utzian series}

In this section, we study extensions of the classical identities given in
\cite[Eqs.\ 4.5.1, 4.5.2 and 4.3.4]{Ba}. We begin with a general formula
that extends \cite[Eq.\ 4.3.1]{Ba}: 

\begin{Proposition}
\label{11P1}
Let
\begin{equation}
\label{1e11P1}
\gamma
=\frac{p(a-p)(b+c-a)}{bc-p(a-p)}.
\end{equation}
Then
\begin{eqnarray}
\label{2e11P1}
&&{}_{r+6}F_{s+4} \left( \left. 
{\displaystyle a,b,c,a-p+1,p+1,a_1,\ldots,a_r,-n
\atop \displaystyle 1+a-b,1+a-c,p,a-p,b_1,\ldots,b_s}\right| 
x\right) \nonumber\\
&&=\sum_{m=0}^n \left(
\frac{\left(\frac{a}{2}\right)_m\left(\frac{a+1}{2}\right)_m(a-b-c)_m(\gamma+1)_m(a_1)_m\cdots(a_r)_m(-n)_m(-4x)^m}
{m!(1+a-b)_m(1+a-c)_m(\gamma)_m(b_1)_m\cdots(b_s)_m}\right.\nonumber\\
&&\times \left. 
{}_{r+2}F_{s} \left( \left. 
{\displaystyle a+2m,a_1+m,\ldots,a_r+m,-n+m
\atop \displaystyle b_1+m,\ldots,b_s+m}\right| 
x\right) \right).
\end{eqnarray} 
\end{Proposition}

\begin{proof}
Using (\ref{e3R1C1P18}), we have
\begin{eqnarray*}
&&{}_{r+6}F_{s+4} \left( \left. 
{\displaystyle a,b,c,a-p+1,p+1,a_1,\ldots,a_r,-n
\atop \displaystyle 1+a-b,1+a-c,p,a-p,b_1,\ldots,b_s}\right| 
x\right) \\
&&=\sum_{k=0}^n 
\frac{(a)_k(b)_k(c)_k(a-p+1)_k(p+1)_k(a_1)_k\cdots (a_r)_k(-n)_kx^k}
{k!(1+a-b)_k(1+a-c)_k(p)_k(a-p)_k(b_1)_k\cdots (b_s)_k}\\
&&=\sum_{k=0}^n \left(
\frac{(a)_k(a_1)_k\cdots (a_r)_k(-n)_kx^k}
{k!(b_1)_k\cdots (b_s)_k}\right.\\
&&\left.\times {}_4F_3 \left( \left. 
{\displaystyle a-b-c,\gamma+1,a+k,-k
\atop \displaystyle 1+a-b,1+a-c,\gamma}\right| 
1\right)\right),
\end{eqnarray*}
where
\begin{equation*}
\gamma=\frac{p(a-p)(b+c-a)}
{bc-p(a-p)}.
\end{equation*} 

We write the ${}_4F_3$ series on the right-hand side above as
a summation, switch the order of summation in the resulting
expression, and then simplify to obtain (\ref{2e11P1}).
\end{proof}

\begin{Remark}
\label{1R11P1}
Formula (\ref{2e11P1}) is an extension of 
\cite[Eq.\ 4.3.1]{Ba}. In fact, \cite[Eq.\ 4.3.1]{Ba}
follows from (\ref{2e11P1}) by letting $x=1$ and $p \to \infty$.
\end{Remark}

We next use Proposition \ref{11P1} to obtain a generalization of 
\cite[Eqs.\ 4.5.1 and 4.5.2]{Ba}:

\begin{Proposition}
\label{3P16}
We have
\begin{eqnarray}
\label{1e3P16}
&&{}_7F_6 \left( \left. 
{\displaystyle a,b,c,a-p+1,p+1,q+1,-n
\atop \displaystyle 1+a-b,1+a-c,p,a-p,q,w}\right| 
1\right)\\
&&=\frac{(w-a-1)_n(\alpha+1)_n}
{(w)_n(\alpha)_n}\nonumber\\
&&\times
{}_7F_6 \left( \left. 
{\displaystyle 1+a-w,\frac{a}{2},\frac{a+1}{2},a-b-c,\beta+1,\gamma+1,-n
\atop \displaystyle 1+a-b,1+a-c,\frac{2+a-w-n}{2},\frac{3+a-w-n}{2},\beta,\gamma}\right| 
1\right),\nonumber 
\end{eqnarray}
where
\begin{equation}
\label{2e3P16}
\alpha=\frac{q(1+a-w)}{a-q},
\end{equation}
\begin{equation}
\label{3e3P16}
\beta=\frac{q(1+a-w)+n(a-q)}{1+2q-w+n}
\end{equation}
and
\begin{equation}
\label{4e3P16}
\gamma=\frac{p(a-p)(b+c-a)}{bc-p(a-p)}.
\end{equation}
\end{Proposition}

\begin{proof}
Use $q+1,q,w,1$ for $a_1,b_1,b_2,x$, respectively, in (\ref{2e11P1}) to obtain
\begin{eqnarray*}
&&{}_{7}F_{6} \left( \left. 
{\displaystyle a,b,c,a-p+1,p+1,q+1,-n
\atop \displaystyle 1+a-b,1+a-c,p,a-p,q,w}\right| 
1\right) \\
&&=\sum_{m=0}^n \left(
\frac{\left(\frac{a}{2}\right)_m\left(\frac{a+1}{2}\right)_m(a-b-c)_m(\gamma+1)_m(q+1)_m(-n)_m(-4)^m}
{m!(1+a-b)_m(1+a-c)_m(\gamma)_m(q)_m(w)_m}\right.\nonumber\\
&&\times \left. 
{}_{3}F_{2} \left( \left. 
{\displaystyle a+2m,q+1+m,-n+m
\atop \displaystyle q+m,w+m}\right| 
1\right) \right),
\end{eqnarray*} 
where
\begin{equation*}
\gamma
=\frac{p(a-p)(b+c-a)}{bc-p(a-p)}.
\end{equation*}

Sum the ${}_3F_2$ series on the right-hand side above according to (\ref{1e3f2}) 
and simplify to obtain the result.
\end{proof}

We note that 
the ${}_7F_6$ series on the left-hand side of (\ref{1e3P16}) deviates from
a well-poised series in two pairs of numerator and denominator parameters while
the ${}_7F_6$ series on the right-hand side of (\ref{1e3P16}) is Saalsch\"utzian.

Letting $q \to \infty$ in (\ref{1e3P16}), we obtain the following special case:

\begin{Corollary}
\label{1C3P16}
We have
\begin{eqnarray}
\label{1e1C3P16}
&&{}_6F_5 \left( \left. 
{\displaystyle a,b,c,a-p+1,p+1,-n
\atop \displaystyle 1+a-b,1+a-c,p,a-p,w}\right| 
1\right)\\
&&=\frac{(w-a)_n}
{(w)_n}\nonumber\\
&&\times
{}_6F_5 \left( \left. 
{\displaystyle 1+a-w,\frac{a}{2},\frac{a+1}{2},a-b-c,\gamma+1,-n
\atop \displaystyle 1+a-b,1+a-c,\frac{1+a-w-n}{2},\frac{2+a-w-n}{2},\gamma}\right| 
1\right),\nonumber 
\end{eqnarray}  
where
\begin{equation}
\label{2e1C3P16}
\gamma=\frac{p(a-p)(b+c-a)}{bc-p(a-p)}.
\end{equation}
\end{Corollary}

We remark that (\ref{1e1C3P16}) is the special case $k=1$
of \cite[Theorem 7.1(ii)]{Maier} as well as
\cite[Cor.\ 4]{WR}.

On the other hand, letting $p \to \infty$ in (\ref{1e3P16}), we obtain the following result:

\begin{Corollary}
\label{2C3P16}
We have
\begin{eqnarray}
\label{1e2C3P16}
&&{}_5F_4 \left( \left. 
{\displaystyle a,b,c,q+1,-n
\atop \displaystyle 1+a-b,1+a-c,q,w}\right| 
1\right)\\
&&=\frac{(w-a-1)_n(\alpha+1)_n}
{(w)_n(\alpha)_n}\nonumber\\
&&\times
{}_6F_5 \left( \left. 
{\displaystyle 1+a-w,\frac{a}{2},\frac{a+1}{2},1+a-b-c,\beta+1,-n
\atop \displaystyle 1+a-b,1+a-c,\frac{2+a-w-n}{2},\frac{3+a-w-n}{2},\beta}\right| 
1\right),\nonumber 
\end{eqnarray}  
where
\begin{equation}
\label{2e2C3P16}
\alpha=\frac{q(1+a-w)}{a-q}
\end{equation}
and
\begin{equation}
\label{3e2C3P16}
\beta=\frac{q(1+a-w)+n(a-q)}{1+2q-w+n}.
\end{equation}
\end{Corollary}

We note that Corollary \ref{2C3P16} generalizes two well-known results of
Whipple and Bailey
given in \cite[Eqs.\ 4.5.1 and 4.5.2]{Ba}. 
Indeed, we have the following:

\begin{enumerate}[label=(\alph*)]

\item Letting $q \to \infty$ in (\ref{1e2C3P16}) gives \cite[Eq.\ 4.5.1]{Ba}
(originally found by Whipple in \cite{Whipple5}).

\item Letting $q=a/2$ in (\ref{1e2C3P16}) gives \cite[Eq.\ 4.5.2]{Ba}
(originally found by Bailey in \cite{Bailey3}).

\end{enumerate}

We next show how Proposition \ref{11P1} along with the result in
Corollary \ref{2C3P16} lead to a formula that generalizes transformations
of both nearly-poised and very-well-poised series to 
Saalsch\"utzian series.

\begin{Proposition}
\label{11P2}
We have
\begin{eqnarray}
\label{1e11P2}
&&{}_9F_8 \left( \left. 
{\displaystyle a,b,c,d,e,a-p+1,p+1,q+1,-n
\atop \displaystyle 1+a-b,1+a-c,1+a-d,1+a-e,p,a-p,q,w}\right| 
1\right)\nonumber\\
&&=\frac{(w-a-1)_n(\alpha+1)_n}{(w)_n(\alpha)_n}\nonumber\\
&&\times \sum_{k=0}^n
\left(
\frac{(-n)_k\left(\frac{a}{2}\right)_k\left(\frac{a+1}{2}\right)_k(1+a-w)_k(1+a-d-e)_k(\beta+1)_k}
{k!(1+a-d)_k(1+a-e)_k\left(\frac{2+a-w-n}{2}\right)_k\left(\frac{3+a-w-n}{2}\right)_k(\beta)_k}\right.\nonumber\\
&&\times \left. {}_5F_4 \left( \left. 
{\displaystyle -k,a-b-c,d,e,\gamma+1
\atop \displaystyle 1+a-b,1+a-c,d+e-a-k,\gamma}\right| 
1\right)\right),
\end{eqnarray} 
where
\begin{equation}
\label{2e11P2}
\alpha=\frac{q(1+a-w)}{a-q},
\end{equation}
\begin{equation}
\label{3e11P2}
\beta=\frac{q(1+a-w)+n(a-q)}{1+2q-w+n}
\end{equation}
and
\begin{equation}
\label{4e11P2}
\gamma=\frac{p(a-p)(b+c-a)}{bc-p(a-p)}.
\end{equation}
\end{Proposition}

\begin{proof}
Use $d,e,q+1$ for $a_1,a_2,a_3$, respectively,
$1+a-d,1+a-e,q,w$ for $b_1,b_2,b_3,b_4$, respectively, 
and $x=1$ in
(\ref{2e11P1}), and then apply
(\ref{1e2C3P16}) to write the ${}_5F_4$ series on the right-hand side
as a Saalsch\"utzian ${}_6F_5$ series. 
After that reverse the order of summation
and simplify.
\end{proof}

The formula in Proposition \ref{11P2} is a very general one. It extends both
very-well-poised identities as well as nearly-poised identities. In fact,
letting $q \to a/2$ first in (\ref{1e11P2}) and then letting $w \to 1+a+n$ in the resulting
formula yields 
\begin{eqnarray}
\label{1e1C11P2}
&&{}_9F_8 \left( \left. 
{\displaystyle a,1+\frac{a}{2},b,c,d,e,a-p+1,p+1,-n
\atop \displaystyle \frac{a}{2},1+a-b,1+a-c,1+a-d,1+a-e,p,a-p,1+a+n}\right| 
1\right)\nonumber\\
&&=\frac{(1+a)_n(1+a-d-e)_n}{(1+a-d)_n(1+a-e)_n}\nonumber\\
&&\times \left. {}_5F_4 \left( \left. 
{\displaystyle a-b-c,d,e,\gamma+1,-n
\atop \displaystyle 1+a-b,1+a-c,d+e-a-n,\gamma}\right| 
1\right)\right),
\end{eqnarray} 
where
\begin{equation}
\label{2e1C11P3}
\gamma=\frac{p(a-p)(b+c-a)}{bc-p(a-p)},
\end{equation}
which is the very-well-poised ${}_9F_8(1)$ to Saalsch\"utzian ${}_5F_4(1)$
transformation found by Srivastava, Vyas and Fatawat in 
\cite[Theorem 3.2]{SVF} that generalizes 
the classical
Whipple's transformation of
a very-well-poised ${}_7F_6(1)$ series to a Saalsch\"utzian ${}_4F_3(1)$
series (see \cite{Whipple2}, \cite{Whipple4} and \cite[Eq.\ 4.3.4]{Ba}). 
On the other hand, letting $b \to \infty$ in 
(\ref{1e11P2}) 
and then letting $c \to \infty$ in the resulting formula
leads to (\ref{1e3P16}), which greatly generalizes the classical
nearly-poised to Saalsch\"utzian transformations found by Whipple and Bailey
(see \cite[Eqs.\ 4.5.1 and 4.5.2]{Ba}).

\section{Extensions of hypergeometric transformations of 
Saalsch\"utzian to
Saalsch\"utzian series}

In this section, we extend the well-known Whipple transform 
(see \cite{Whipple2}, \cite{Whipple3} and \cite[Eq.\ 7.2.1]{Ba}) 
which involves two terminating Saalsch\"utzian
${}_4F_3(1)$ series. We begin with the following general result:

\begin{Proposition}
\label{12P1}
Let
\begin{equation}
\label{1e12P1}
\gamma
=\frac{p(c-a-1)(c-b-1)}{ab+p(c-a-b-1)}.
\end{equation}
Then
\begin{eqnarray}
\label{2e12P1}
&&{}_{r+4}F_{s+2} \left( \left. 
{\displaystyle a,b,p+1,a_1,\ldots,a_r,-n
\atop \displaystyle c,p,b_1,\ldots,b_s}\right| 
x\right) \nonumber\\
&&=\sum_{m=0}^n \left(
\frac{(c-a-1)_m(c-b-1)_m(\gamma+1)_m(a_1)_m\cdots(a_r)_m(-n)_mx^m}
{m!(c)_m(\gamma)_m(b_1)_m\cdots(b_s)_m}\right.\nonumber\\
&&\times \left. 
{}_{r+2}F_{s} \left( \left. 
{\displaystyle 1+a+b-c,a_1+m,\ldots,a_r+m,-n+m
\atop \displaystyle b_1+m,\ldots,b_s+m}\right| 
x\right) \right).
\end{eqnarray} 
\end{Proposition}

\begin{proof}
Using (\ref{e5R1C1P18}), we have
\begin{eqnarray*}
&&{}_{r+4}F_{s+2} \left( \left. 
{\displaystyle a,b,p+1,a_1,\ldots,a_r,-n
\atop \displaystyle c,p,b_1,\ldots,b_s}\right| 
x\right) \\
&&=\sum_{k=0}^n 
\frac{(a)_k(b)_k(p+1)_k(a_1)_k\cdots (a_r)_k(-n)_kx^k}
{k!(c)_k(p)_k(b_1)_k\cdots (b_s)_k}\\
&&=\sum_{k=0}^n \left(
\frac{(1+a+b-c)_k(a_1)_k\cdots (a_r)_k(-n)_kx^k}
{k!(b_1)_k\cdots (b_s)_k}\right.\\
&&\left.\times {}_4F_3 \left( \left. 
{\displaystyle c-a-1,c-b-1,\gamma+1,-k
\atop \displaystyle c,\gamma,c-a-b-k}\right| 
1\right)\right),
\end{eqnarray*}
where
\begin{equation*}
\gamma=\frac{p(c-a-1)(c-b-1)}
{ab+p(c-a-b-1)}.
\end{equation*} 

We write the ${}_4F_3$ series on the right-hand side above as
a summation, switch the order of summation in the resulting
expression, and then simplify to obtain (\ref{2e12P1}).
\end{proof}

The extension of the Whipple transform is given next:

\begin{Proposition}
\label{12P2}
We have
\begin{eqnarray}
\label{1e12P2}
&&{}_6F_5 \left( \left. 
{\displaystyle a,b,c,p+1,q+1,-n
\atop \displaystyle d,e,f,p,q}\right| 
1\right)\nonumber\\
&&=\frac{(e-c-1)_n(f-c-1)_n(\alpha+1)_n}{(e)_n(f)_n(\alpha)_n}\nonumber\\
&&\times {}_6F_5 \left( \left. 
{\displaystyle d-a-1,d-b-1,c,\gamma+1,\delta+1,-n
\atop \displaystyle d,2+c-e-n,2+c-f-n,\gamma,\delta}\right| 
1\right),
\end{eqnarray} 
where
\begin{equation}
\label{2e12P2}
d+e+f-a-b-c+n=3,
\end{equation}
\begin{equation}
\label{3e12P2}
\alpha=\frac{q(e-c-1)(f-c-1)}{(c-q)(d-a-b-1)},
\end{equation}
\begin{equation}
\label{4e12P2}
\gamma
=\frac{p(d-a-1)(d-b-1)}{ab+p(d-a-b-1)}
\end{equation}
and
\begin{equation}
\label{5e12P2}
\delta=\frac{q(e-c-1)(f-c-1)+n(c-q)(d-a-b-1)}{(e-c-1)(f-c-1)-(c-q)(d-a-b-1)}.
\end{equation}

\end{Proposition}

\begin{proof}
Use $d,c,q+1,e,f=3+a+b+c-d-e-n,q,1$ for $c,a_1,a_2,b_1,b_2,b_3,x$, respectively,
in
(\ref{2e12P1}), and then sum the  Saalsch\"utzian
${}_4F_3(1)$ series on the right-hand side according to
(\ref{e1R1C1P18}).
\end{proof}

The relation in (\ref{1e12P2}) involves two Saalsch\"utzian
${}_6F_5(1)$ series in each of which series 
two numerator parameters exceed two 
denominator parameters by one. This relation is a generalization of
the classical Whipple transform (see \cite{Whipple2}, \cite{Whipple3} and \cite[Eq.\ 7.2.1]{Ba}) 
involving two terminating Saalsch\"utzian
${}_4F_3(1)$ series as we show after Corollary \ref{1C12P2} below.

\begin{Remark}
\label{1R12P2}
Let
\begin{eqnarray}
\label{1e1R12P2}
&&\tilde{F}_n(a,b,c;d,e,f;p,q)\\
&&=(d)_n(e)_n(f)_n(\alpha)_n
{}_6F_5 \left( \left. 
{\displaystyle a,b,c,p+1,q+1,-n
\atop \displaystyle d,e,f,p,q}\right| 
1\right),\nonumber
\end{eqnarray}
where
\begin{equation*}
d+e+f-a-b-c+n=3
\end{equation*}
and $\alpha$ is as given in (\ref{3e12P2}).
Then equation (\ref{1e12P2}) implies that
\begin{eqnarray}
\label{2e1R12P2}
&&\tilde{F}_n(a,b,c;d,e,f;p,q)\\
&&=(-1)^n
\tilde{F}_n(d-a-1,d-b-1,c;d,2+c-e-n,2+c-f-n;\gamma,\delta),\nonumber
\end{eqnarray}
where $\gamma$ and $\delta$ are as given in
(\ref{4e12P2}) and (\ref{5e12P2}), respectively.
\end{Remark}

\begin{Corollary}
\label{1C12P2}
We have
\begin{eqnarray}
\label{1e1C12P2}
&&{}_5F_4 \left( \left. 
{\displaystyle a,b,c,p+1,-n
\atop \displaystyle d,e,f,p}\right| 
1\right)\nonumber\\
&&=\frac{(e-c)_n(f-c)_n}{(e)_n(f)_n}\nonumber\\
&&\times {}_5F_4 \left( \left. 
{\displaystyle d-a-1,d-b-1,c,\gamma+1,-n
\atop \displaystyle d,1+c-e-n,1+c-f-n,\gamma}\right| 
1\right),
\end{eqnarray} 
where
\begin{equation}
\label{2e1C12P2}
d+e+f-a-b-c+n=2,
\end{equation}
and
\begin{equation}
\label{3e1C12P2}
\gamma
=\frac{p(d-a-1)(d-b-1)}{ab+p(d-a-b-1)}.
\end{equation}
\end{Corollary}

\begin{proof}
Let $q \to c$
in
(\ref{1e12P2}) 
and then replace $c+1$ with $c$.
\end{proof}

The two ${}_5F_4(1)$ series in (\ref{1e1C12P2}) are both  
Saalsch\"utzian and in each one of them a numerator parameter exceeds a denominator
parameter by one. Letting $p=b$ in (\ref{1e1C12P2}) gives 
the Whipple transform involving two Saalsch\"utzian
${}_4F_3(1)$ series.

\begin{Corollary}
\label{2C12P2}
We have
\begin{eqnarray}
\label{1e2C12P2}
&&{}_4F_3 \left( \left. 
{\displaystyle a,c,p+1,-n
\atop \displaystyle d,e,p}\right| 
1\right)\\
&&=\frac{(e-c)_n}{(e)_n}
{}_4F_3 \left( \left. 
{\displaystyle d-a-1,c,\gamma+1,-n
\atop \displaystyle d,1+c-e-n,\gamma}\right| 
1\right),\nonumber
\end{eqnarray} 
where
\begin{equation}
\label{2e2C12P2}
\gamma
=\frac{p(d-a-1)}{p-a}.
\end{equation}
\end{Corollary}

\begin{proof}
In (\ref{1e1C12P2}), fix $a,c,d,e,p$ and $n$, and let
\begin{equation*}
f=2+a+b+c-d-e-n
\end{equation*}
depend on $b$. Let $b \to \infty$ to obtain the result.
\end{proof}

We note that (\ref{1e2C12P2}) generalizes the classical relation 
involving two
terminating ${}_3F_2(1)$ series 
(see Sheppard \cite{Sheppard} and Whipple \cite{Whipple1} which follow
Thomae \cite{T}). 
Indeed, letting $p \to \infty$ in 
(\ref{1e2C12P2}) gives 
\begin{eqnarray}
\label{3e2C12P2}
&&{}_3F_2 \left( \left. 
{\displaystyle a,c,-n
\atop \displaystyle d,e}\right| 
1\right)\\
&&=\frac{(e-c)_n}{(e)_n}
{}_3F_2 \left( \left. 
{\displaystyle d-a,c,-n
\atop \displaystyle d,1+c-e-n}\right| 
1\right).\nonumber
\end{eqnarray} 

Equation (\ref{1e2C12P2}) also extends the Saalsch\"utzian
${}_4F_3(1)$ summation (\ref{e1R1C1P18}). In fact, (\ref{e1R1C1P18})
follows from (\ref{1e2C12P2}) by setting $e=2+a+c-d-n$ and then
summing the resulting ${}_3F_2(1)$ series on the right-hand side
according to (\ref{1e3f2}).

\section{Extensions of hypergeometric
transformations of
very-well-poised 
and nearly-poised series to very-well-poised
series}

In this section, we extend the relation between two 
terminating very-well-poised ${}_{11}F_{10}(1)$ series
given by Srivastava, Vyas and Fatawat in 
\cite[Theorem 3.4]{SVF} (which generalizes Bailey's
${}_9F_8$ transformation in \cite[Eq.\ 4.3.7]{Ba})
to a relation between two
terminating very-well-poised ${}_{13}F_{12}(1)$ series.
We also extend the formulas found in
\cite[Eqs.\ 4.5.3, 4.5.4 and 4.5.5]{Ba}.
We begin with a general formula that extends
\cite[Eq.\ 4.3.6]{Ba}:

\begin{Proposition}
\label{13P1}
If 
\begin{equation}
\label{1e13P1}
\lambda=2a-b-c-d, 
\end{equation}
then
\begin{eqnarray}
\label{2e13P1}
&&{}_{r+7}F_{s+5} \left( \left. 
{\displaystyle a,b,c,d,a-p+1,p+1,a_1,\ldots,a_r,-n
\atop \displaystyle 1+a-b,1+a-c,1+a-d,p,a-p,b_1,\ldots,b_s}\right| 
x\right) \nonumber\\
&&=\sum_{m=0}^n \left(
\frac{(\lambda)_m(\lambda+b-a)_m(\lambda+c-a)_m(\lambda+d-a)_m\left(\frac{a}{2}\right)_m\left(\frac{a+1}{2}\right)_m}
{m!\left(\frac{\lambda}{2}\right)_m\left(\frac{\lambda+1}{2}\right)_m(1+a-b)_m(1+a-c)_m(1+a-d)_m}\right.\nonumber\\
&&\times
\frac{\left(\frac{\lambda}{2}-\gamma+1\right)_m\left(\frac{\lambda}{2}+\gamma+1\right)_m(a_1)_m\cdots(a_r)_m(-n)_mx^m}
{\left(\frac{\lambda}{2}+\gamma\right)_m\left(\frac{\lambda}{2}-\gamma\right)_m(b_1)_m\cdots(b_s)_m}\\
&&\times \left. 
{}_{r+3}F_{s+1} \left( \left. 
{\displaystyle a+2m,a-\lambda,a_1+m,\ldots,a_r+m,-n+m
\atop \displaystyle 1+\lambda+2m,b_1+m,\ldots,b_s+m}\right| 
x\right) \right),\nonumber
\end{eqnarray} 
where
\begin{equation}
\label{3e13P1}
\gamma^2
=\frac{\lambda^2}{4}
-\frac{p(a-p)(a-b-c)(a-b-d)(a-c-d)}
{bcd+p(a-p)(a-b-c-d)}.
\end{equation}
\end{Proposition}

\begin{proof}
Using (\ref{1e2R3C11P2}), we have
\begin{eqnarray*}
&&{}_{r+7}F_{s+5} \left( \left. 
{\displaystyle a,b,c,d,a-p+1,p+1,a_1,\ldots,a_r,-n
\atop \displaystyle 1+a-b,1+a-c,1+a-d,p,a-p,b_1,\ldots,b_s}\right| 
x\right) \\
&&=\sum_{k=0}^n 
\frac{(a)_k(b)_k(c)_k(d)_k(a-p+1)_k(p+1)_k}
{k!(1+a-b)_k(1+a-c)_k(1+a-d)_k(p)_k(a-p)_k}\\
&&\times
\frac{(a_1)_k\cdots (a_r)_k(-n)_kx^k}
{(b_1)_k\cdots (b_s)_k}\\
&&=\sum_{k=0}^n \left(
\frac{(a)_k(a-\lambda)_k(a_1)_k\cdots (a_r)_k(-n)_kx^k}
{k!(1+\lambda)_k(b_1)_k\cdots (b_s)_k}\right.\\
&&{}_9F_8 \left(  
{\lambda,1+\frac{\lambda}{2},\lambda+b-a,\lambda+c-a,\lambda+d-a,
\atop 
\frac{\lambda}{2},1+a-b,1+a-c,1+a-d,}
\right.\\
&&\left. \left. \left.
{a+k,\frac{\lambda}{2}-\gamma+1,\frac{\lambda}{2}+\gamma+1,-k;
\atop 
1+\lambda-a-k,\frac{\lambda}{2}+\gamma,\frac{\lambda}{2}-\gamma,1+\lambda+k;}
\right| 1\right)\right),
\end{eqnarray*}
where
\begin{equation*}
\lambda=2a-b-c-d
\end{equation*}
and
\begin{equation*}
\gamma^2
=\frac{\lambda^2}{4}
-\frac{p(a-p)(a-b-c)(a-b-d)(a-c-d)}
{bcd+p(a-p)(a-b-c-d)}.
\end{equation*} 

We write the ${}_9F_8$ series on the right-hand side above as
a summation, switch the order of summation in the resulting
expression, and then simplify to obtain (\ref{2e13P1}).
\end{proof}

\begin{Remark}
\label{1R13P1}
Formula (\ref{2e13P1}) is an extension of 
\cite[Eq.\ 4.3.6]{Ba}. In fact, \cite[Eq.\ 4.3.6]{Ba}
follows from (\ref{2e13P1}) by letting $x=1$ and $p=d$.
\end{Remark}

We now obtain the generalization of 
Srivastava, Vyas and Fatawat's ${}_{11}F_{10}$
transformation given in \cite[Theorem 3.4]{SVF}:

\begin{Proposition}
\label{13P3}
Suppose
\begin{equation}
\label{1e13P3}
3a=b+c+d+e+f+g-n. 
\end{equation}
Then
\begin{eqnarray}
\label{2e13P3}
&&{}_{13}F_{12} \left( 
{\displaystyle a,1+\frac{a}{2},b,c,d,e,f,
\atop 
\displaystyle \frac{a}{2},1+a-b,1+a-c,1+a-d,1+a-e,1+a-f,}
\right.\nonumber\\
&& \left.\left.
{\displaystyle g,a-p+1,p+1,a-q+1,q+1,-n
\atop 
\displaystyle 1+a-g,p,a-p,q,a-q,1+a+n}
\right| 
1\right)\\
&&=\frac{(1+a)_n(1+\lambda-e)_n(1+\lambda-f)_n(1+\lambda-g)_n}
{(1+\lambda)_n(1+a-e)_n(1+a-f)_n(1+a-g)_n}\nonumber\\
&&\times \frac{\left(\frac{\mu}{2}-\delta+1\right)_n\left(\frac{\mu}{2}+\delta+1\right)_n}
{\left(\frac{\mu}{2}+\delta\right)_n\left(\frac{\mu}{2}-\delta\right)_n}\nonumber\\
&&\times {}_{13}F_{12} \left( 
{\displaystyle \lambda,1+\frac{\lambda}{2},\lambda+b-a,\lambda+c-a,\lambda+d-a,e,f,
\atop 
\displaystyle \frac{\lambda}{2},1+a-b,1+a-c,1+a-d,1+\lambda-e,1+\lambda-f,}
\right.\nonumber\\
&& \left.\left.
{\displaystyle g,\frac{\lambda}{2}-\gamma+1,\frac{\lambda}{2}+\gamma+1,
\frac{\lambda}{2}-\epsilon+1,\frac{\lambda}{2}+\epsilon+1,-n
\atop 
\displaystyle 1+\lambda-g,\frac{\lambda}{2}+\gamma,\frac{\lambda}{2}-\gamma,
\frac{\lambda}{2}+\epsilon,\frac{\lambda}{2}-\epsilon,1+\lambda+n}
\right| 
1\right),\nonumber
\end{eqnarray} 
where
\begin{equation}
\label{3e13P3}
\lambda=2a-b-c-d,
\end{equation}
\begin{equation}
\label{4e13P3}
\mu=2a-e-f-g,
\end{equation}
\begin{equation}
\label{5e13P3}
\gamma^2
=\frac{\lambda^2}{4}
-\frac{p(a-p)(a-b-c)(a-b-d)(a-c-d)}
{bcd+p(a-p)(a-b-c-d)},
\end{equation}
\begin{equation}
\label{6e13P3}
\delta^2
=\frac{\mu^2}{4}
-\frac{q(a-q)(a-e-f)(a-e-g)(a-f-g)}
{efg+q(a-q)(a-e-f-g)}
\end{equation}
and
\begin{eqnarray}
\label{7e13P3}
&&\epsilon^2
=\frac{\lambda^2}{4}\\
&&-\,\frac{\left[\displaystyle{q(a-q)(a-e-f)(a-e-g)(a-f-g) \atop +\,n(\mu+n)(efg+q(a-q)(a-e-f-g))}\right]}
{\left[\displaystyle{(a-e-f)(a-e-g)(a-f-g) \atop -\,(\mu+n)(ef+eg+fg+a(a-e-f-g)-q(a-q))}\right]}.\nonumber
\end{eqnarray}
\end{Proposition}

\begin{proof}
Use $1+\frac{a}{2},e,f,g,a-q+1,q+1$ for $a_1,a_2,a_3,a_4,a_5,a_6$, respectively, 
$\frac{a}{2},1+a-e,1+a-f,1+a-g,q,a-q,1+a+n$ for 
$b_1,b_2,b_3,b_4,b_5,b_6,b_7$, respectively, 
and $x=1$ in
(\ref{2e13P1}), and then apply
(\ref{1e1R3C11P2}) to sum the ${}_9F_8(1)$ 
series on the right-hand side.
The result follows after some simplification.
\end{proof}

Equation (\ref{2e13P3}) above involves two 
terminating very-well-poised ${}_{13}F_{12}(1)$ series.
It generalizes the result of Srivastava, Vyas and Fatawat \cite[Theorem 3.4]{SVF} between two
terminating very-well-poised ${}_{11}F_{10}(1)$ series, which in turn is a generalization
of Bailey's ${}_9F_8$ transformation 
(see \cite[Eq.\ 4.3.7]{Ba}). Indeed, \cite[Theorem 3.4]{SVF} follows from 
our result (\ref{2e13P3}) upon setting $q=e$.

We next extend  the formulas found in
\cite[Eqs.\ 4.5.3, 4.5.4 and 4.5.5]{Ba}.
First, we obtain an even more general result:

\begin{Proposition}
\label{13P4}
We have
\begin{eqnarray}
\label{1e13P4}
&&{}_8F_7 \left( \left. 
{\displaystyle a,b,c,d,a-p+1,p+1,q+1,-n
\atop \displaystyle 1+a-b,1+a-c,1+a-d,p,a-p,q,w}\right| 
1\right)\\
&&=\frac{(2\lambda-a)_n(\lambda-a)_n(\alpha+1)_n}{(1+\lambda)_n(2\lambda-2a)_n(\alpha)_n}\nonumber\\
&&\times {}_{13}F_{12} \left( 
{\displaystyle \lambda,1+\frac{\lambda}{2},\frac{a}{2},\frac{a+1}{2},\lambda+b-a,\lambda+c-a,\lambda+d-a,
\atop 
\displaystyle \frac{\lambda}{2},\frac{2+2\lambda-a}{2},\frac{1+2\lambda-a}{2},1+a-b,1+a-c,1+a-d,}
\right.\nonumber\\
&& \left.\left.
{\displaystyle 1+a-w,\frac{\lambda}{2}-\gamma+1,\frac{\lambda}{2}+\gamma+1,
\frac{\lambda}{2}-\delta+1,\frac{\lambda}{2}+\delta+1,-n
\atop 
\displaystyle \lambda+w-a,\frac{\lambda}{2}+\gamma,\frac{\lambda}{2}-\gamma,
\frac{\lambda}{2}+\delta,\frac{\lambda}{2}-\delta,1+\lambda+n}
\right| 
1\right),\nonumber
\end{eqnarray} 
where
\begin{equation}
\label{2e13P4}
\lambda=2a-b-c-d,
\end{equation}
\begin{equation}
\label{3e13P4}
w=1+2a-2\lambda-n, 
\end{equation}
\begin{equation}
\label{4e13P4}
\alpha=\frac{q(2\lambda-a)}{2q-a},
\end{equation}
\begin{equation}
\label{5e13P4}
\gamma^2
=\frac{\lambda^2}{4}
-\frac{p(a-p)(a-b-c)(a-b-d)(a-c-d)}
{bcd+p(a-p)(a-b-c-d)}
\end{equation}
and
\begin{equation}
\label{6e13P4}
\delta^2
=\frac{\lambda^2}{4}
-\frac{q(2\lambda-a)+n(2q-a)}
{2}.
\end{equation}
\end{Proposition}

\begin{proof}
Use $q+1,q,w,1$ for $a_1,b_1,b_2,x$, respectively, in
(\ref{2e13P1}) 
(where $w$ is as given in (\ref{3e13P4}))
and then apply
(\ref{e1R1C1P18}) to sum
the  Saalsch\"utzian
${}_4F_3(1)$ series on the right-hand side.
The final result follows after some simplification.
\end{proof}

We note that
the terminating ${}_8F_7(1)$ series on the left-hand side
of (\ref{1e13P4}) is Saalsch\"utzian (i.e. with parametric excess 
$\omega=1$) and deviates from a well-poised series
in two pairs of numerator and denominator parameters.
The terminating ${}_{13}F_{12}(1)$ series on the right-hand side
of (\ref{1e13P4}) is very-well-poised.

The special case of Proposition \ref{13P4}
given in the next corollary 
is a direct extension of
the results
found in \cite[Eqs.\ 4.5.3, 4.5.4 and 4.5.5]{Ba}:

\begin{Corollary}
\label{1C13P4}
We have
\begin{eqnarray}
\label{1e1C13P4}
&&{}_6F_5 \left( \left. 
{\displaystyle a,b,c,d,q+1,-n
\atop \displaystyle 1+a-b,1+a-c,1+a-d,q,w}\right| 
1\right)\\
&&=\frac{(2\lambda-a)_n(\lambda-a)_n(\alpha+1)_n}{(1+\lambda)_n(2\lambda-2a)_n(\alpha)_n}\nonumber\\
&&\times {}_{11}F_{10} \left( 
{\displaystyle \lambda,1+\frac{\lambda}{2},\frac{a}{2},\frac{a+1}{2},\lambda+b-a,\lambda+c-a,\lambda+d-a,
\atop 
\displaystyle \frac{\lambda}{2},\frac{2+2\lambda-a}{2},\frac{1+2\lambda-a}{2},1+a-b,1+a-c,1+a-d,}
\right.\nonumber\\
&& \left.\left.
{\displaystyle 1+a-w,
\frac{\lambda}{2}-\delta+1,\frac{\lambda}{2}+\delta+1,-n
\atop 
\displaystyle \lambda+w-a,
\frac{\lambda}{2}+\delta,\frac{\lambda}{2}-\delta,1+\lambda+n}
\right| 
1\right),\nonumber
\end{eqnarray} 
where
\begin{equation}
\label{2e1C13P4}
\lambda=1+2a-b-c-d,
\end{equation}
\begin{equation}
\label{3e1C13P4}
w=1+2a-2\lambda-n, 
\end{equation}
\begin{equation}
\label{4e1C13P4}
\alpha=\frac{q(2\lambda-a)}{2q-a}
\end{equation}
and
\begin{equation}
\label{5e1C13P4}
\delta^2
=\frac{\lambda^2}{4}
-\frac{q(2\lambda-a)+n(2q-a)}
{2}.
\end{equation}
\end{Corollary}

\begin{proof}
Let $p=b$ in (\ref{1e13P4}) and then 
replace $b+1$ with $b$.
\end{proof}

Equation (\ref{1e1C13P4}) above expresses a certain 
terminating Saalsch\"utzian (i.e. with parametric excess $\omega=1$)
${}_6F_5(1)$
series that deviates from a well-poised series in two pairs 
of numerator and denominator parameters
in terms of a terminating very-well-poised 
${}_9F_8(1)$ series.
This equation is a direct generalization of
the classical results found by Bailey in 
\cite[Eqs.\ 8.1, 8.2 and 8.3]{Bailey3}
and reproduced in
\cite[Eqs.\ 4.5.3, 4.5.4 and 4.5.5]{Ba}
that transform terminating nearly-poised ${}_5F_4(1)$ series
with parametric excesses $\omega=1$ and $\omega=2$ and
a terminating very-well-poised ${}_6F_5(1)$ series 
with parametric excess $\omega=1$ in terms of
terminating very-well-poised ${}_9F_8(1)$ series. 
Indeed, we have the following:

\begin{enumerate}[label=(\alph*)]

\item
Letting $q \to -n$ in (\ref{1e1C13P4}) gives
\cite[Eq.\ 4.5.3]{Ba}.

\item
Letting $q \to \frac{a}{2}$ in (\ref{1e1C13P4}) gives
\cite[Eq.\ 4.5.4]{Ba}.

\item
Letting $q \to \infty$ in (\ref{1e1C13P4}) gives
\cite[Eq.\ 4.5.5]{Ba}.

\end{enumerate} 

\section{Extensions of classical quadratic transformations}

In this last section, we show how by taking certain limits
 of the relations in Section 3, we can obtain generalizations
of some classical quadratic transformations of hypergeometric functions.
Two of the classical quadratic transformations are the following:
\begin{eqnarray}
\label{1e6}
&&{}_3F_2 \left( \left. 
{\displaystyle a,b,c
\atop \displaystyle 1+a-b,1+a-c}\right| 
x\right)\\
&&=
(1-x)^{-a}
{}_3F_2 \left( \left. 
{\displaystyle \frac{a}{2},\frac{a+1}{2},1+a-b-c
\atop \displaystyle 1+a-b,1+a-c}\right| 
-\frac{4x}{(1-x)^2}\right)\nonumber 
\end{eqnarray}  
and
\begin{eqnarray}
\label{2e6}
&&{}_4F_3 \left( \left. 
{\displaystyle a,1+\frac{a}{2},b,c
\atop \displaystyle \frac{a}{2},1+a-b,1+a-c}\right| 
x\right)\\
&&=
(1+x)(1-x)^{-a-1}
{}_3F_2 \left( \left. 
{\displaystyle \frac{a+1}{2},\frac{a+2}{2},1+a-b-c
\atop \displaystyle 1+a-b,1+a-c}\right| 
-\frac{4x}{(1-x)^2}\right).\nonumber 
\end{eqnarray}  

The transformation (\ref{1e6}) is due to Whipple \cite{Whipple5}
and (\ref{2e6}) is due to Bailey \cite{Bailey3}.
The transformation (\ref{2e6}) is sometimes referred to as the 
companion transformation
of (\ref{1e6}) (see \cite[Section 4]{GS}, for example).

In this section, 
we shall obtain quadratic transformations that extend both (\ref{1e6}) and its companion
(\ref{2e6}) in terms of 
single transformations (see (\ref{1e6P1}) and (\ref{1e2C6P1}) below).
We begin with our most general extension which follows as 
a consequence of Proposition \ref{3P16} from Section 3. 

\begin{Proposition}
\label{6P1}
The following quadratic transformation holds:
\begin{eqnarray}
\label{1e6P1}
&&{}_6F_5 \left( \left. 
{\displaystyle a,b,c,a-p+1,p+1,q+1
\atop \displaystyle 1+a-b,1+a-c,p,a-p,q}\right| 
x\right)\\
&&=\left(1+\left(\frac{a-q}{q}\right)x\right)
(1-x)^{-a-1}\nonumber\\
&&\times
{}_5F_4 \left( \left. 
{\displaystyle \frac{a}{2},\frac{a+1}{2},a-b-c,\gamma+1,\delta+1
\atop \displaystyle 1+a-b,1+a-c,\gamma,\delta}\right| 
-\frac{4x}{(1-x)^2}\right),\nonumber 
\end{eqnarray}  
where
\begin{equation}
\label{2e6P1}
\gamma=\frac{p(a-p)(b+c-a)}{bc-p(a-p)},
\end{equation}
and
\begin{equation}
\label{3e6P1}
\delta=\frac{q+(a-q)x}{1+x}.
\end{equation}
\end{Proposition}

\begin{proof}
Let $w=-n/x$ in (\ref{1e3P16})
and then let $n \to \infty$.
\end{proof}

The quadratic transformation (\ref{1e6P1}) above is a very general one.
It extends both (\ref{1e6}) and (\ref{2e6}) as we show after Corollary \ref{2C6P1}
below. Furthermore, the presence of the variable $x$ 
as part of $\delta$ that appears
in the numerator and denominator
parameters in the ${}_5F_4$ series on the right-hand side of (\ref{1e6P1}) seems to be a new feature
and distinguishes this transformation from most other known transformations.

Two interesting quadratic transformations that follow as
special cases of (\ref{1e6P1}) are derived below.

\begin{Corollary}
\label{1C6P1}
The following quadratic transformation holds:
\begin{eqnarray}
\label{1e1C6P1}
&&{}_5F_4 \left( \left. 
{\displaystyle a,b,c,a-p+1,p+1
\atop \displaystyle 1+a-b,1+a-c,p,a-p}\right| 
x\right)\\
&&=(1-x)^{-a}\nonumber\\
&&\times
{}_4F_3 \left( \left. 
{\displaystyle \frac{a}{2},\frac{a+1}{2},a-b-c,\gamma+1
\atop \displaystyle 1+a-b,1+a-c,\gamma}\right| 
-\frac{4x}{(1-x)^2}\right),\nonumber 
\end{eqnarray}  
where
\begin{equation}
\label{2e1C6P1}
\gamma=\frac{p(a-p)(b+c-a)}{bc-p(a-p)}.
\end{equation}
\end{Corollary}

\begin{proof}
Let $q \to \infty$ in (\ref{1e6P1}).
\end{proof}

We note that the transformation (\ref{1e1C6P1}) is the same as the special case
$k=1$ in \cite[Theorem $3.4$]{Maier}.

\begin{Corollary}
\label{2C6P1}
The following quadratic transformation holds:
\begin{eqnarray}
\label{1e2C6P1}
&&{}_4F_3 \left( \left. 
{\displaystyle a,b,c,q+1
\atop \displaystyle 1+a-b,1+a-c,q}\right| 
x\right)\\
&&=\left(1+\left(\frac{a-q}{q}\right)x\right)
(1-x)^{-a-1}\nonumber\\
&&\times
{}_4F_3 \left( \left. 
{\displaystyle \frac{a}{2},\frac{a+1}{2},1+a-b-c,\delta+1
\atop \displaystyle 1+a-b,1+a-c,\delta}\right| 
-\frac{4x}{(1-x)^2}\right),\nonumber 
\end{eqnarray}  
where
\begin{equation}
\label{2e2C6P1}
\delta=\frac{q+(a-q)x}{1+x}.
\end{equation}
\end{Corollary}

\begin{proof}
Let $p \to \infty$ in (\ref{1e6P1}).
\end{proof}

The quadratic transformation (\ref{1e2C6P1}) above is the direct
extension of (\ref{1e6}) and its companion
(\ref{2e6}). In fact, Whipple's transformation
(\ref{1e6}) follows from (\ref{1e2C6P1}) by letting
$q \to \infty$,
and Bailey's companion transformation (\ref{2e6}) follows 
from (\ref{1e2C6P1}) by letting
$q \to \frac{a}{2}$.

\begin{Corollary}
\label{3C6P1}
We have, if $q \neq a/2$,
\begin{eqnarray}
\label{1e3C6P1}
&&{}_6F_5 \left( \left. 
{\displaystyle a,b,c,a-p+1,p+1,q+1
\atop \displaystyle 1+a-b,1+a-c,p,a-p,q}\right| 
-1\right)\\
&&=\frac{(2q-a)2^{-a-1}}{q}\nonumber\\
&&\times
{}_4F_3 \left( \left. 
{\displaystyle \frac{a}{2},\frac{a+1}{2},a-b-c,\gamma+1
\atop \displaystyle 1+a-b,1+a-c,\gamma}\right| 
1\right),\nonumber 
\end{eqnarray}  
where
\begin{equation}
\label{2e3C6P1}
\gamma=\frac{p(a-p)(b+c-a)}{bc-p(a-p)}.
\end{equation}
\end{Corollary}

\begin{proof}
Let $x \to -1$ in (\ref{1e6P1}).
\end{proof}

The formula in (\ref{1e3C6P1}) above is a generalization of Whipple's formula
(see \cite{Whipple2,Whipple5})
\begin{eqnarray}
\label{3e6}
&&{}_3F_2 \left( \left. 
{\displaystyle a,b,c
\atop \displaystyle 1+a-b,1+a-c}\right| 
-1\right)\\
&&=2^{-a}
{}_3F_2 \left( \left. 
{\displaystyle \frac{a}{2},\frac{a+1}{2},1+a-b-c
\atop \displaystyle 1+a-b,1+a-c}\right| 
1\right).\nonumber 
\end{eqnarray}  
In fact, (\ref{3e6}) follows from (\ref{1e3C6P1}) by 
letting first $q \to \infty$ and then 
letting $p \to \infty$ in the resulting equation.

\end{document}